\documentclass[11pt]{amsart}
\usepackage{amssymb,latexsym,epsfig,color}

\newcommand{\R}{{\mathbb R}}

\newcommand{\Set}{{\mathcal S}}

\newtheorem{theorem}{Theorem}
\newtheorem{proposition}[theorem]{Proposition}
\newtheorem{corollary}[theorem]{Corollary}
\newtheorem{lemma}[theorem]{Lemma}

\theoremstyle{definition}
\newtheorem{definition}[theorem]{Definition}

\title[Bachet's Weights Problem]
{Bachet's Problem: as few weights to weigh them all}
\author{Edwin O'Shea}
\thanks{Supported by Science Foundation Ireland's Mathematics Initiative.}
\address{School of Mathematical Sciences, University College, Cork, Ireland}
\email{oshea.edwin@gmail.com}
\date{\today}

\begin{document}

\maketitle


\vspace{-.1in}

The genesis of many areas in mathematics can often be found in some simply-put puzzle, a word problem 
that doesn't require any formal language or concise definitions to understand. 
A few cases in point are graph theory having its origins in Euler's 
{\em Bridges of K{\" o}nigsberg problem}, the {\em Chinese 
remainder problem} which best captures the rules of modular arithmetic in number 
theory \& abstract algebra, and the dark arts of probability having their roots in 
17$^{\textup{th}}$-century games of chance. Generalizations of these problems form 
the bedrock for much of what came afterward. 

In other subjects, progress is made instead with the root problems 
leading to others but without those root problems ever being solved. This is the case with 
number theory's first problems like Goldbach's conjecture and the twin primes conjecture. 
But it can also be the case that the first problem of a modern and active area of 
mathematics can simply be forgotten as that, even if the problem enjoys an enduring popularity 
both within and outside the classroom. This is certainly the case with the problem that we will 
generalize here and which we argue should be regarded as one of the first problems, if not the 
first, of the thoroughly modern area of integer partitions: 

\begin{center}
{\em What is the least number of pound weights that can be used on a scale pan to 
weigh any integral number of pounds from 1 to 40 inclusive, if the weights 
can be placed in either of the scale pans~?} 
\end{center}

W.W. Rouse Ball \cite[pp.50]{Bal} attributes the first recording of this problem 
to Bachet in the early 17th century, calling it {\em Bachet's Weights Problem}, and 
Hardy \& Wright thought it fit to include it in their wonderful and highly influential 
{\em An introduction to the theory of numbers} \cite{HarWri}. 
However, Bachet's problem, as noted by Knobloch~\cite{Kno}, stretches all the way back to 
Fibonacci \cite[{\em On IIII Weights Weighing Forty Pounds}]{Fib} in 1202! 

\parbox{3.2in}{\includegraphics[scale=0.50]{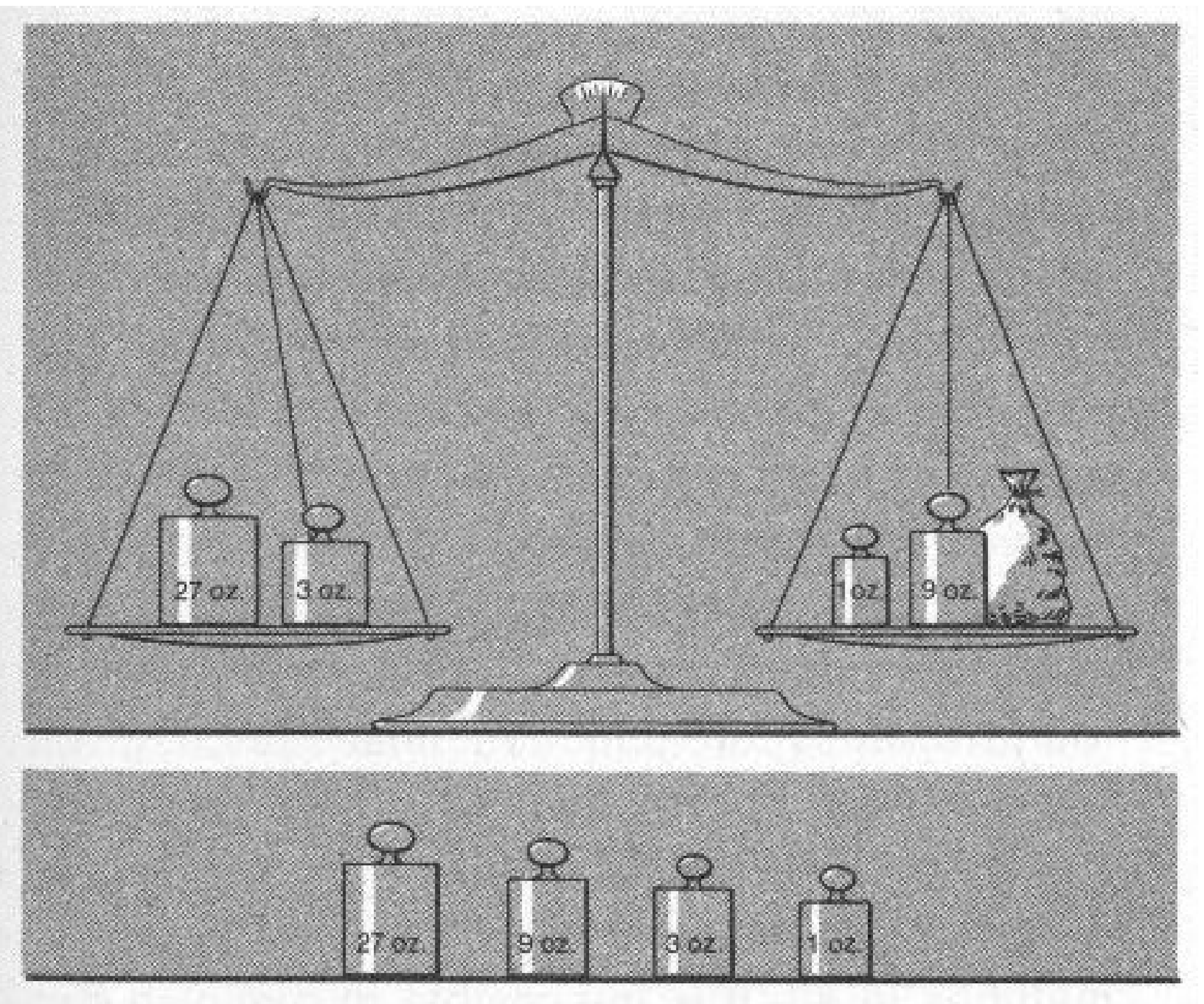}}
\parbox{2.8in}{\includegraphics[scale=0.62]{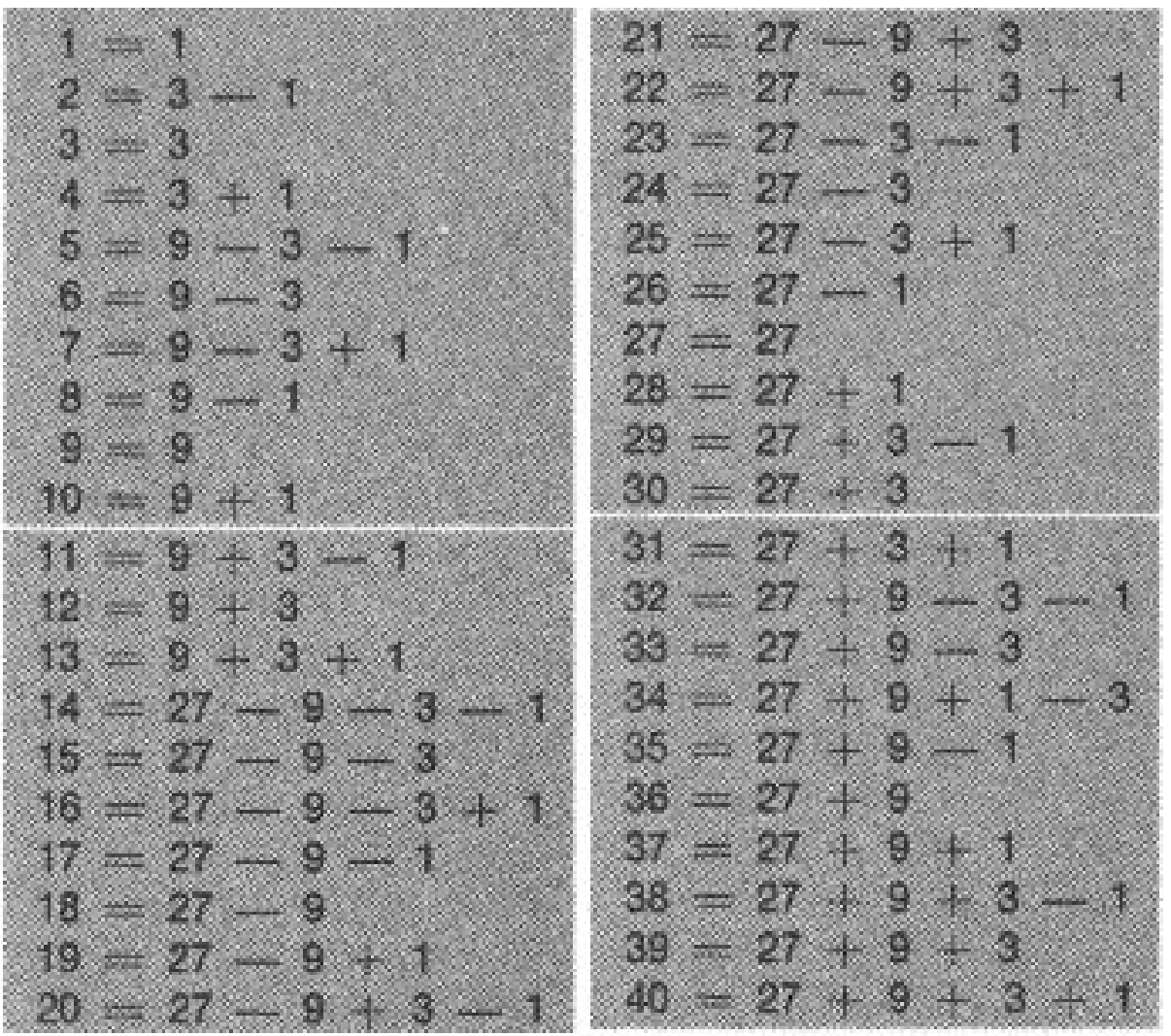}}

Bachet's problem needs no more than four weights and these (unique) pound 
weights are 1,3,9 and 27. The figure \cite[pp. 53]{Ste} displays how to weigh 20 
(Steinhaus \cite{Ste} had the good sense to only lift ounce rather than pound weights onto the page) 
and the table 
(also from \cite[pp. 53]{Ste}) displays how to measure all the weights between 1 and 40 
inclusive, a positive coefficient assigned to weights placed on the left scale, a negative 
to those on the right. Writing the solution as an integer partition with four parts 
$40=1+3+9+27$, Bachet's problem's noble roots in {\em Fibonacci's Liber Abaci} \cite{Fib} 
make it a viable candidate for the first problem of integer partitions. 

Until relatively recently the only known generalizations of this problem were that of 
replacing $40$ with integers of the form $\frac{1}{2}(3^{n+1}-1)$ \cite[\S 9.7]{HarWri} 
and the appropriate partition, as we might guess at this juncture, being powers of $3$. 
This has received some practical attention in economics \cite{Tesler} as it provides optimal 
denominations of coins and currency. However, a retort to this \cite{VanHove} is that in 
our common decimal system not everyone can think quickly in ternary.

The {\em generalized Bachet's problem} that we will explore here is that of finding 
appropriate weights when one replaces $40$ with any positive integer. 
The full generalization, due to Park \cite{Par} and studied further by R{\o}dseth \cite{Rod}, 
not only tells us the {\em minimum number of parts} needed when $40$ is replaced by any $m$ but 
{\em all possible ways to accordingly break up} a given $m$. Furthermore, 
we can also {\em count} the number of distinct ways to break up such an $m$. 
For example, when we replace $40$ by $m=25$ 
we'll still need no more than four parts but there are now nine ways to break up $25$ to solve 
Bachet's problem. Written as partitions with four parts, these are: 
$$
\begin{array}{cccccccccc}
 25 & & = & 1+3+9+12 & & = & 1+3+8+13 & & = & 1+3+7+14 \,  \\
    & & = & 1+3+6+15 & & = & 1+3+5+16 & & = & 1+3+4+17 \,  \\
    & & = & 1+2+7+15 & & = & 1+2+6+16 & & = & 1+2+5+17 
\end{array}
$$
Remarkably, given the age and popularity of Bachet's problem, these headways have 
come to light only in the last fifteen or so years and they seem to be 
little known at that.  Given its status as one of the first problems of partitions 
of integers, we aim to rectify this sad state of affairs 
and to do so in a lively and informal yet unambiguous fashion, using only our sharp wits 
and a willingness to induct! We also hope to introduce impressionable readers to some of 
the wonders of partitions of integers, recurrence relations, generating functions and 
counting integer points in polyhedra.

We will also expound on similar problems like the following: 
{\em what is the least number of pound weights that can be used on a scale pan to 
weigh any integral number of pounds from 1 to 15 inclusive, if the weights 
can be placed in only one of the scale pans~?}  Finally, we will close with MacMahon's 
generalization of (the two-scale) Bachet's problem: he noticed \cite{MacQuar} 
that $1,3,9,27$ can be used to {\em uniquely} weigh every integer weight between 
$1$ and $40$. For example, the figure displays that $20 = -1+3-9+27$ and we claim, 
in the sense of Bachet, that this is the only way to write $20$ using $1,3,9$ and $27$.
We will see what the factorization $3 \times 3 \times3 \times 3$ 
of $81$ has to do with the weight set $1,3,9,27$ for $40$.

\section*{A First Solution to the Generalized Bachet's Problem}
Before becoming a touch more formal, let's provide a taster of what's to come by 
providing our first candidates, one candidate of mostly ternary weights for each positive 
integer $m$, to solve the generalized Bachet's problem. 
Given a positive integer $m$ there is a unique integer $n$ such that 
$\frac{1}{2}(3^n-1)+1 \leq m \leq \frac{1}{2}(3^{n+1}-1)$. We can 
break the integer $m$ into $n+1$ smaller integer weights consisting of 
those elements in the multi-set 
${\mathcal W}_m := \{ 1,3,3^2,\ldots,3^{n-1}, m - (1+3+3^2+ \cdots +3^{n-1})\}$. 
For example, if $m=25$ then 
$$
14 = \frac{1}{2}(3^3-1)+1 \leq 25 \leq \frac{1}{2}(3^{3+1}-1) = 40 
\, \, \textup{and} \, \, 
{\mathcal W}_{25} = \{ 1,3,9, 25 - (1+3+9) \} = \{ 1,3,9,12 \}.
$$ 
\begin{proposition} \label{pro:first}
Every integer weight $l$ with $0 \leq l \leq m$ can be measured using a two scale 
balance with the weights from the multiset ${\mathcal W}_m$.
\end{proposition}
To see that this is true in the case of $m=25$ observe that every integer in the closed 
interval $[-13,13]$ can be measured using $\{1,3,9\}$. With the extra weight of $12$ we
can in addition measure every integer in the shifted closed interval $12+[-13,13] = [-1,25]$ 
and so the proposition holds for $m=25$. Let's prove it now for every $m$.
\begin{proof}
For $m=1,2,3,4$ (those $m$'s with $n = 0$ or $1$) we have ${\mathcal W}_1 = \{ 1 \}$, 
${\mathcal W}_2 = \{ 1, 1 \}$, ${\mathcal W}_3 = \{ 1, 2 \}$ and ${\mathcal W}_4 = \{ 1, 3 \}$ 
respectively and, for every such $m$, every $0 \leq l \leq m$ can be measured using both pans 
of the two scale balance with the weights in ${\mathcal W}_m$. 
Assume that this is the case for every $m \leq \frac{1}{2}(3^{n}-1)$. 
In particular, assume that 
${\mathcal W}_{\frac{1}{2}(3^{n}-1)} := \{ 1,3,3^2,\ldots,3^{n-1} \}$ 
can be used to weigh every integer $l$ in the closed interval 
$[-\frac{1}{2}(3^{n}-1), \frac{1}{2}(3^{n}-1)]$ -- thinking in terms of the 
two scale pans, a negative $-l$ would have the weights on the scales interchanged 
from that of the positive $l$. 

We will now proceed by induction on $n$ to show that every 
$l \leq m$ can be measured (using both pans of the two scale balance) with the 
weights in ${\mathcal W}_m$ for all $m$'s with 
$\frac{1}{2}(3^n-1)+1 \leq m \leq \frac{1}{2}(3^{n+1}-1)$. 
Since the multiset ${\mathcal W}_{\frac{1}{2}(3^{n}-1)}$ is contained 
in ${\mathcal W}_m$ then, by our inductive hypothesis, every integer 
in the closed interval $[-\frac{1}{2}(3^{n}-1), \frac{1}{2}(3^{n}-1)]$ 
can be measured by using weights from 
${\mathcal W}_m \backslash \{ m - \frac{1}{2}(3^n-1) \}$ on the two scale balance. 
Consequently, every integer weight in the following closed interval 
can be measured using ${\mathcal W}_m$:
$$
m - \frac{1}{2}(3^n-1) \, + \, [-\frac{1}{2}(3^{n}-1), \, \frac{1}{2}(3^{n}-1)] 
\, = \, [m - 3^n + 1, \, m].
$$
When combined with our induction hypothesis, this implies that all integers in the 
union of the closed intervals $[0, \, \frac{1}{2}(3^{n}-1)] \, \cup \, [m - 3^n + 1, \, m]$ 
can be measured using ${\mathcal W}_m$. Now recall that 
$m \leq \frac{1}{2}(3^{n+1}-1)$ which implies that 
$$
m-3^n+1 \, \leq \, \frac{1}{2}(3^{n+1}-1) - 3^n + 1 \, =  \, \frac{1}{2}(3^{n}-1) + 1
$$ 
and so the integers in the set $[0, \, \frac{1}{2}(3^{n}-1)] \, \cup \, [m - 3^n + 1, \, m]$ 
are precisely those integers in the set $[0,m]$. In other words, every integer weight 
$l$ with $0 \leq l \leq m$ can be measured using a two scale balance with the weights from 
${\mathcal W}_m$.
\end{proof}

In the case of $m = \frac{1}{2}(3^{n+1}-1)$, the above proposition was intimated by 
Fibonacci in \cite[{\em On IIII Weights Weighing Forty Pounds}]{Fib} and first proved 
by Hardy \& Wright \cite[\S 9.7]{HarWri} 
who went further by showing that ${\mathcal W}_{\frac{1}{2}(3^{n+1}-1)}$ 
is not only the smallest multiset of weights that satisfy the Bachet problem for 
$m = \frac{1}{2}(3^{n+1}-1)$ but that it is the unique such multiset. 

In the next sections, we will see that ${\mathcal W}_m$ is a multiset of {\em minimal} 
size with the property that every weight between $0$ and $m$ can be measured using a 
two scale balance. From this analysis Hardy \& Wright's 
claim of ${\mathcal W}_{\frac{1}{2}(3^{n+1}-1)}$ being the unique such multiset will follow.
But in order to do so we will need first to delve into the language of partitions of integers.

\section*{Partitions of Integers}
Luckily for us, the description of all solutions to the generalized Bachet problem 
is surprisingly elegant and simple when phrased in terms of  
{\em partitions of integers}. 
A {\em partition} of a positive integer $m$ is an ordered sequence of positive integers 
that sum to $m$: $m=\lambda_0 + \lambda_1 + \lambda_2 + \cdots + \lambda_n$ with 
$\lambda_0 \leq  \lambda_1 \leq \lambda_2 \leq \cdots \leq \lambda_n$. We call 
the $n+1$ $\lambda_i$'s the {\em parts} of the above partition. For example, $5$ has seven 
distinct partitions given by 
$$
5 \, = \, 1+1+1+1+1 \, = \, 1+1+1+2 \, = \, 1+2+2 \, = \, 1+1+3 \, =  1+4 \, \, = \, 2+3
$$
and we denote this by $p(5)=7$. Analogous to the hand-shaking 
lemma in graph theory, the first lemma that everyone encounters in integer partitions is: 
the number of partitions of a given $m$ 
with no parts larger than $n+1$ equals the number of partitions of $m$ with at most $n+1$ parts. For $m=5$ 
and $n+1=2$ this translates to $|\{1+1+1+1+1, \, 1+1+1+2, \, 1+2+2 \}| = |\{ 5, \, 1+4, \, 2+3 \}|$. 
See \cite{AndEri} for a first introduction to integer partitions and \cite{And} for a more advanced 
perspective. 

Returning to Bachet's problem, let's call a partition of $m$ a {\em Bachet partition} if
\begin{center}
(1) every integer $0 \leq l \leq m$ can be written as $l = \sum_{i=0}^n{\beta_i \lambda_i}$ 
where each $\beta_i \in \{-1,0,1\}$
\end{center} 
and (2) there does not exist another partition of $m$ satisfying (1) 
with fewer parts than $n+1$. 

For example, only four of the seven partitions of $5$ satisfy condition (1): 
$\{ 1+1+1+1+1, \, 1+1+1+2, \, 1+2+2, \, 1+1+3 \}$.
And of these four partitions 
only two have the fewest possible number of three parts: $\{ 1+2+2, \, 1+1+3 \}$. 
In short, $5$ has two Bachet partitions. 

Another example is the partition $1 + 3 + 9 + 12$ of $25$, whose parts are 
precisely the elements of ${\mathcal W}_{25}$. 
Proposition~\ref{pro:first} above amounts to saying that this partition 
satisfies condition (1). What remains to be shown is whether this partition satisfies (2). 
We could of course list all $p(25) = 1958$ partitions of $25$ \cite[A000041]{Slo} and check which 
of those satisfy (1). And then pick out those with the fewest number of parts just as we did 
above for finding the Bachet partitions of $5$. But with some simple observations about condition 
(1) above we'll soon be able to do much better than this brute-force, tedious computation.

Noting that condition (1) above involves positive and negative $\beta_i$ coefficients 
it can be beneficial to only have to worry about addition and to do so we can 
rewrite condition (1) as: 
\begin{center}
(1)$^\prime$ every integer $0 \leq l \leq 2m$ can be written as $l = \sum_{i=0}^n{\alpha_i \lambda_i}$ 
where each $\alpha_i \in \{0,1,2\}$.
 \end{center}
The equivalence of conditions (1) \& (1)$^\prime$, as essentially noted by Hardy \& Wright in \cite[\S 9.7]{HarWri}, 
is given by the shift of $m = \lambda_0 + \lambda_1 + \cdots + \lambda_n$ in 
$l-m = \sum_{i=0}^n{\alpha_i \lambda_i} - \sum_{i=0}^n{\lambda_i} = \sum_{i=0}^n{\beta_i \lambda_i}$. 
Note that we could just as easily have replaced $0 \leq l \leq m$ in (1) with 
$-m \leq l \leq m$ since, thinking in terms of the two scales, a negative $-l$ would have the 
weights on the scales interchanged from that of the positive $l$.

Partitions of an integer $m$ satisfying (1)$^\prime$ are called {\em $2$-complete partitions} 
and were introduced by Park \cite{Par} as recently as 1998. 
This shift between conditions (1) and (1)$^\prime$ is little more than a sleight of hand but 
it does resolve the central difficulty in dealing with (1), in that it avoids having to deal with 
both addition \& subtraction operations, whereas (1)$^\prime$ involves only addition. 
We'll see in the next section that condition (1)$^\prime$ immediately tells us that $\lambda_0 = 1$ 
but this is not as obvious when using only (1). Much more will also become transparent from this 
formulation in the next section where we resolve (2), the minimality of parts condition. 

\section*{The Minimality of Parts Condition}

A simple equivalence regarding the $2$-complete partitions, first proved by Park, 
will amazingly tell us all that we need to know about Bachet partitions. We'll 
deal first with the minimality condition (2).
\begin{lemma} \cite{Par} \label{lem:suff}
If $m=\lambda_0 + \lambda_1 + \lambda_2 + \cdots + \lambda_n$ is a $2$-complete partition 
then $\lambda_0 = 1$ and 
$\lambda_i \leq 1 + 2(\lambda_0 + \lambda_1 + \cdots + \lambda_{i-1})$ 
for every $i = 1,2,\ldots, n$.
\end{lemma}
\begin{proof}
Since $0 \leq 1 \leq 2m$ then we must be able to write $1$ as a $\{0,1,2\}$-combination 
of the parts $\lambda_0 \leq  \lambda_1 \leq \lambda_2 \leq \cdots \leq \lambda_n$ 
and if $\lambda_0 \geq 2$ then such a $\{0,1,2\}$-combination of the parts would be impossible.
Hence, $\lambda_0 = 1$ as claimed.

Consider next, for each $i = 1,\ldots,n$, the non-negative integer $\lambda_i -1$. 
Since $\lambda_i -1 < \lambda_i \leq \ldots \leq \lambda_n$, and since 
$m=\lambda_0 + \lambda_1 + \lambda_2 + \cdots + \lambda_n$ is a $2$-complete partition, 
then there must exist a $\{0,1,2\}$-combination of the parts 
$\lambda_0,\lambda_1,\ldots,\lambda_{i-1}$ that equals $\lambda_i -1$. 
Hence $\lambda_i -1$ cannot exceed the largest of all 
$\{0,1,2\}$-combinations of $\lambda_0,\lambda_1,\ldots,\lambda_{i-1}$, which 
would be $2\lambda_0 + 2\lambda_1 + \cdots + 2\lambda_{i-1}$. In other words, 
$\lambda_i \leq 1 + 2(\lambda_0 + \lambda_1 + \cdots + \lambda_{i-1})$ as claimed. 
\end{proof}

\begin{corollary} \label{cor:3s}
If $m=\lambda_0 + \lambda_1 + \lambda_2 + \cdots + \lambda_n$ is a $2$-complete partition 
then $\lambda_i \leq 3^i$ for every $i=0,1,\ldots, n$.
\end{corollary}

This corollary follows by first noting that if $\lambda_0 = 1$ then $\lambda_1 \leq 1 + 2(1) = 3$. 
In turn, $\lambda_2 \leq 1 + 2 (1 + 3) = 9$ and the corollary now follows by an inductive 
argument. Now we come to the minimality condition of Bachet partitions. Corollary~\ref{cor:3s} 
implies that if $m = \lambda_0 + \lambda_1 + \cdots + \lambda_n$ is a Bachet partition then 
the sum of the parts in the partition cannot exceed $\sum_{i=0}^n{3^i} = \frac{1}{2}(3^{n+1}-1)$. 
That is, 
$$
m \leq \frac{1}{2}(3^{n+1}-1) < \frac{1}{2}3^{n+1} 
\, \, \, \, \textup{or} \, \, \, \,   
\textup{log}_3(2m) < n+1.
$$ 
Since $n+1$ is an integer then the integer part of $\textup{log}_3(2m)$ 
i.e.  $\lfloor \textup{log}_3(2m) \rfloor < n+1$ (the function $\lfloor x \rfloor$ 
takes a real number $x$ to the greatest integer that is less than or equal to $x$). 
Since both $\lfloor \textup{log}_3(2m) \rfloor$ and $n+1$ are integers then 
$\lfloor \textup{log}_3(2m) \rfloor \leq n$. 
In summary, Corollary~\ref{cor:3s} tells us that a Bachet partition must have 
{\em at least} $\lfloor \textup{log}_3(2m) \rfloor + 1$ parts. So if we could 
find a partition satisfying condition (1) with exactly 
$\lfloor \textup{log}_3(2m) \rfloor + 1$ parts then $\lfloor \textup{log}_3(2m) \rfloor + 1$ 
must be precisely the number of parts needed for a Bachet partition of $m$. 

But we do have such a partition! The elements of the multiset ${\mathcal W}_m$ from 
Proposition~\ref{pro:first} , reordered in increasing order and set equal (in order) 
to $\lambda_0$ through $\lambda_n$, 
make such a partition. For example, $25 = 1 + 3 + 9 + 12$ is a Bachet partition because 
of Proposition~\ref{pro:first} combined with Corollary~\ref{cor:3s}. 
\begin{theorem} \label{thm:min}
A Bachet partition of a positive integer $m$ has precisely 
$\lfloor \textup{log}_3(2m) \rfloor + 1$ parts.
\end{theorem}

This theorem was essentially stated in \cite{Par} and formally stated, including 
Proposition~\ref{pro:first}, by R{\o}dseth \cite[Lemma 3.2]{Rod} where Bachet 
partitions are called {\em minimal 2-complete partitions}. 
One might wonder next: are the Bachet partitions from Proposition~\ref{pro:first} the only 
Bachet partitions for each positive integer $m$~? In the case of $m = \frac{1}{2}(3^{n+1}-1)$ 
it is now easy to show that the answer is yes. To see this let 
$\frac{1}{2}(3^{n+1}-1) = \lambda_0 + \lambda_1 + \cdots + \lambda_n$ be a Bachet partition. 
If any of the $\lambda_j$'s were strictly less than $3^j$ then, from Corollary~\ref{cor:3s}, 
we would have $\frac{1}{2}(3^{n+1}-1) = \sum_{i = 0}^n{\lambda_i} < \frac{1}{2}(3^{n+1}-1)$ 
which cannot occur. Hence, as claimed by \cite[\S 9.7]{HarWri}, $1 + 3 + 3^2 + \cdots + 3^n$ 
is the unique Bachet partition for $m = \frac{1}{2}(3^{n+1}-1)$. 

In the next section, we will show that a partition is a Bachet partition if and only if 
it both has the number of parts as stated above and, amazingly, the conclusion of 
Lemma~\ref{lem:suff} is satisfied for all parts in the partition. 
But before doing so permit us to digress a little and say what was so enjoyable about this section: 
we discovered everything we needed to know about the number of parts needed for a Bachet partition 
by starting with a very simple collection of inequalities (Lemma~\ref{lem:suff}) and then we used 
a very generous version of these inequalities to attain $\lambda_i \leq 3^i$. When combined 
with Proposition~\ref{pro:first} we were able to solve the problem of the number of weights 
needed for the Bachet problem. 
We should know 
better but it is still surprising to attain meaningful, sharp results from languid 
inequalities such as those used in this section. 
See \cite{Steele} for a delightful, analysis-flavored account on all things being ``inequal''!

\section*{Bachet Partitions as Lattice Points in Polyhedra}

Recall our example of $25 = 1 + 3 + 9 + 12$ as a Bachet partition. In contrast 
to the scenario where $\frac{1}{2}(3^{n+1}-1) = 1 + 3 + 3^2 + \cdots + 3^n$ is a unique 
Bachet partition for that particular $m$, there are many (many being nine!) Bachet 
partitions for $m=25$: 
$$
\begin{array}{cccccccccc}
 25 & & = & 1+3+9+12 & & = & 1+3+8+13 & & = & 1+3+7+14 \,  \\
    & & = & 1+3+6+15 & & = & 1+3+5+16 & & = & 1+3+4+17 \,  \\
    & & = & 1+2+7+15 & & = & 1+2+6+16 & & = & 1+2+5+17 
\end{array}
$$
That these partitions are precisely the Bachet partitions for $25$ follow from this remarkable result: 
\begin{theorem} \label{thm:necc} \textup{(Park \cite[Theorem 2.2]{Par})} 
The partition $m=\lambda_0 + \lambda_1 + \lambda_2 + \cdots + \lambda_n$ is a Bachet 
partition if and only if $n = \lfloor \textup{log}_3(2m) \rfloor$, $\lambda_0 = 1$ 
and 
$\lambda_i \leq 1 + 2(\lambda_0 + \lambda_1 + \cdots + \lambda_{i-1})$ 
for every $i = 1,2,\ldots, n$.
\end{theorem}
\begin{proof}
Due to Lemma~\ref{lem:suff} and Theorem~\ref{thm:min} all we need show is that 
if $\lambda_i \leq 1 + 2(\lambda_0 + \lambda_1 + \cdots + \lambda_{i-1})$ 
for every $i = 1,2,\ldots, n$ then $m=\lambda_0 + \lambda_1 + \lambda_2 + \cdots + \lambda_n$ 
is a $2$-complete partition. 
This will be carried out by induction on the number of parts in the partition. Let $\Set_n$ be the 
set of all partitions with $n+1$ parts that satisfy $\lambda_0 = 1$ and 
$\lambda_i \leq 1 + 2(\lambda_0 + \lambda_1 + \cdots + \lambda_{i-1})$ for every $i = 1,2,\ldots, n$. 

We will show that $\Set_n$ is contained in the set of $2$-complete partitions. 
Clearly this is true for $\Set_0 = \{ 1 \}$ and $\Set_1 = \{ 1+1, 1+2, 1+3 \}$. Assume it is so 
for all $\Set_i$'s where $i \leq n-1$. We will show that $\Set_n$ is contained in the 
set of $2$-complete partitions. 
Let $\lambda_0 + \lambda_1 + \lambda_2 + \cdots + \lambda_n$ be a fixed partition  
in $\Set_n$. Note that this implies that 
$\lambda_0 + \lambda_1 + \lambda_2 + \cdots + \lambda_{n-1}$ is in $\Set_{n-1}$ and so 
our inductive hypothesis tells us the following: every integer $l$ less than or equal 
to $2(\lambda_0 + \lambda_1 + \lambda_2 + \cdots + \lambda_{n-1})$ can be written 
as a $\{ 0,1,2 \}$-combination of $\lambda_0, \lambda_1, \lambda_2, \ldots, \lambda_{n-1}$. 

So we can assume from here that we fix $l \leq 2 \sum_{j=0}^{n}{\lambda_j}$ and $l > 2 \sum_{j=0}^{n-1}{\lambda_j}$. 
In this case there will always exist an $1 \leq \alpha_n \leq 2$ such that 
$
(\alpha_n - 1)\lambda_n + 2 \sum_{j=0}^{n-1}{\lambda_j} < l  \leq \alpha_n \lambda_n + 2 \sum_{j=0}^{n-1}{\lambda_j}
$ 
or $\alpha_n = \left\lceil \frac{l - 2 \sum_{j=0}^{n-1}{\lambda_j} }{\lambda_n} \right\rceil$. 
But since $l - \alpha_n \lambda_n \leq 2 \sum_{j=0}^{n-1}{\lambda_j}$ our inductive hypothesis tells us that 
$l - \alpha_n\lambda_n$ can be written as a $\{ 0,1,2 \}$-combination of $\lambda_0, \lambda_1, \lambda_2, \ldots, \lambda_{n-1}$ 
and so $l$ can be written as a $\{ 0,1,2 \}$-combination of $\lambda_0, \lambda_1, \lambda_2, \ldots, \lambda_{n-1}, \lambda_n$.
\end{proof}

One striking aspect of this inequality formulation of the Bachet partitions is that 
for each positive $m$ we can think of the Bachet partitions as the set of 
lattice points (points all of whose entries are integers) 
in the polyhedron in $\R^n$ defined by the inequalities of Theorem~\ref{thm:necc}. 
The nine Bachet partitions of $25$, written as $(\lambda_0, \lambda_1, \lambda_2, \lambda_3) \in \R^4$, 
sit in the (two-dimensional) plane living in $\R^4$ cut out by the equations 
$\lambda_0=1$ and $\lambda_0+\lambda_1+\lambda_2+\lambda_3 = 25$ and by the six additional 
{\em halfspaces} defined by the six 
inequalities: $\lambda_i \leq 1 + 2 (\lambda_0 + \cdots + \lambda_{i-1})$ for each $i = 1,2,3$ 
and $\lambda_0 \leq \lambda_1 \leq \lambda_2 \leq \lambda_3$. By {\em cut out} we really mean 
the region in $\R^4$ given by the intersection of the two three-dimensional planes and the 
six halfspaces: 
\begin{center}
\begin{picture}(0,0)%
\includegraphics{twoScale25.pstex}%
\end{picture}%
\setlength{\unitlength}{2486sp}%
\begingroup\makeatletter\ifx\SetFigFont\undefined%
\gdef\SetFigFont#1#2#3#4#5{%
  \reset@font\fontsize{#1}{#2pt}%
  \fontfamily{#3}\fontseries{#4}\fontshape{#5}%
  \selectfont}%
\fi\endgroup%
\begin{picture}(9892,3854)(675,-5733)
\put(3511,-5416){\makebox(0,0)[lb]{\smash{{\SetFigFont{7}{8.4}{\rmdefault}{\mddefault}{\updefault}{\color[rgb]{0,0,0}$\lambda_2 \leq 3 + 2\lambda_1$}%
}}}}
\put(8596,-2626){\makebox(0,0)[lb]{\smash{{\SetFigFont{7}{8.4}{\rmdefault}{\mddefault}{\updefault}{\color[rgb]{0,0,0}$\lambda_1 \geq 1$}%
}}}}
\put(2836,-4381){\makebox(0,0)[lb]{\smash{{\SetFigFont{7}{8.4}{\rmdefault}{\mddefault}{\updefault}{\color[rgb]{0,0,0}$\lambda_1 \leq 3$}%
}}}}
\put(7606,-5101){\makebox(0,0)[lb]{\smash{{\SetFigFont{7}{8.4}{\rmdefault}{\mddefault}{\updefault}{\color[rgb]{0,0,0}$\lambda_1 \leq \lambda_2$}%
}}}}
\put(1171,-2716){\makebox(0,0)[lb]{\smash{{\SetFigFont{7}{8.4}{\rmdefault}{\mddefault}{\updefault}{\color[rgb]{0,0,0}$\lambda_2 \leq \lambda_3$}%
}}}}
\put(4276,-2446){\makebox(0,0)[lb]{\smash{{\SetFigFont{7}{8.4}{\rmdefault}{\mddefault}{\updefault}{\color[rgb]{0,0,0}$\lambda_3 \leq 3+2\lambda_1 + 2\lambda_2$}%
}}}}
\end{picture}%

\end{center}
The first three inequalities (along with the two-dimensional 
plane) cut out the shaded triangle shown. 
In this case of $m=25$, the other three inequalities that define the ordering of 
the parts of the partitions are not needed -- they are said to be {\em redundant} -- as they do not contribute 
to the {\em cutting out} of the shaded triangle. The Bachet partitions of $25$, as expected from 
Theorem~\ref{thm:necc}, are precisely the integer points in the shaded region.

Another striking consequence of the inequality formulation is that 
every Bachet partition has an {\em hereditary property}: it can be both 
projected down to, and lifted up from, another Bachet partition. In the next sections 
we'll use this hereditary property to count the number of Bachet partitions for 
a given $m$ but in order to do so we must first talk about ternary partitions and 
generating functions.

\section*{Precursor to Counting: Ternary Partitions and Generating Functions}

There is a formula due to R{\o}dseth \cite[Theorem 2.1]{Rod} for counting 
precisely the number of distinct Bachet partitions for a given $m$. It is a 
{\em generating function} formula and 
a quick perusal of the encyclopedic \cite{And} should convince the reader that generating 
functions are {\em the} standard way of counting in the theory of partitions of integers. 
However, the full derivation \cite[\S 4]{Rod} of R{\o}dseth's formula is 
difficult and technical and is beyond the 
scope (and against the informal spirit) of this present article. 
But we will describe R{\o}dseth's formula nonetheless and justify it for a substantial 
number of cases. To describe it we'll first need to talk 
about {\em ternary partitions} and their generating function.
Recall that we can formally write the geometric series 
$1 + (x^t)^1 + (x^{t})^2 + (x^{t})^3 + (x^{t})^4 + \cdots$ as $\frac{1}{1-x^t}$. We define the 
generating function 
$$
F(x) := \sum_{k=0}^{\infty}{f(k)x^k} = \prod_{i=0}^{\infty}\frac{1}{1-x^{3^i}}.
$$ 
where $f(k)$ is understood as the coefficient of $x^k$ in the infinite product 
$$
(1 + x^1 + (x^1)^2 + (x^1)^3 + (x^1)^4 + \cdots)
(1 + x^3 + (x^3)^2 + (x^3)^3 + (x^3)^4 + \cdots)
(1 + x^9 + (x^9)^2 + (x^9)^3 + (x^9)^4 + \cdots)
\cdots
$$ 
The significance of the term {\em generating function} comes from each $f(k)$ counting the 
$k^{\textup{th}}$ instance of some combinatorial phenomenon; in this case, the number of 
partitions of $k$ into powers of $3$ (these are called {\em ternary partitions}).
For example, 
$f(15) = 9$ since there are precisely nine partitions of $15$ all of whose parts 
are powers of $3$: 
\begin{center}
$\overbrace{1+ \cdots +1}^{15 \, \textup{times}}$, 
$\overbrace{1+ \cdots +1}^{12 \, \textup{times}} + 3$, 
$\overbrace{1+ \cdots +1}^{9 \, \textup{times}} + 3+3$, 
$1+1+1+1+1+1+3+3+3$, 
$1+1+1+3+3+3+3$, 
$1+1+1+1+1+1+9$, $1+1+1+3+9$, $3+3+3+3+3$, $3+3+9$.
\end{center}
A contribution of ``$1$'' is made to the coefficient $f(15) = 9$ for each ternary partition of 
$15$. One such contribution would be given by the term $(x^1)^3 (x^3)^4 = x^3x^{12} = x^{15}$ 
which represents the ternary partition $15 = 1+1+1+3+3+3+3$ of three $1$'s and four $3$'s. 
By convention, $f(0)=1$.

The generating function $F(x)$ also satisfies the functional equation $F(x) = \frac{1}{(1-x)}F(x^3)$ 
or, \newline \noindent $(1-x)F(x) = F(x^3)$ and looking at the coefficient of $x^{3k}$ in this 
equation we attain a recurrence relation $f(3k) - f(3k-1) = f(k)$ or 
$$f(3k) = f(3k-1) + f(k).$$ 
Returning to our ternary partitions this recurrence should not be so surprising: it says that the ternary 
partitions of $3k$ can be made from those of $3k-1$ (all of these already contain at least two $1$'s 
as parts and so adding another part equal to $1$ gives all possible ternary partitions of $3k$ with some 
parts equal to $1$) and from those ternary partitions of $k$ (by multiplying all terms of these ternary 
partitions of $k$ by $3$ we get ternary partions of $3k$ with no parts equal to $1$). The recurrence relation 
$f(3k) = f(3k-1) + f(k)$ explains this manner of counting the ternary partitions of $3k$ in a concise 
and unfussy manner. 

In other words, the generating function $F(x)$ is not only an 
accounting mechanism for ternary partitions but we can also manipulate the properties of 
$F(x)$ to recover 
encoded information about the ternary partitions themselves. These are some of the reasons that generating 
function formulae are thought of as the most useful means of counting not only specific partitions of 
integers but other combinatorial phenomena. 
A wonderfully colorful yet precise introduction to 
generating functions in general is \cite{Concrete} and \cite{AndEri} introduces 
them in the context of partitions of integers.

Returning again to our recurrence relations for the ternary partitions, we can observe that 
that $f(3k) = f(3k+1) = f(3k+2)$ since the ternary partitions of $3k+1$ and $3k+2$ 
are those given by adding one and two extra parts equal to $1$ respectively to those of $3k$. 
We can thus generalize the recurrence relation $f(3k) = f(3k-1) + f(k)$ to 
$f(k) = f(k-3) + f(\lfloor \frac{k}{3} \rfloor)$. By the same recurrence we have 
$f(k-3) = f(k-6) + f(\lfloor \frac{k-3}{3} \rfloor) = f(k-6) + f(\lfloor \frac{k}{3} \rfloor -1)$ 
and repeating we have 
$$
f(k) = \sum_{i=0}^{\lfloor \frac{k}{3} \rfloor}{f(i)}
$$
with the initial condition of $f(0)=1$. As we would expect, this recurrence yields 
$f(2)=f(1)=f(0)=1$ and so $f(5)=f(4) = f(3) = f(1)+f(0)=2$ which yields 
$f(15) = 1+1+1+2+2+2 = 9$ as claimed from the generating function above.

\section*{Counting Bachet Partitions: Projecting Down and Lifting Up}

Let us now explain the link between Bachet partitions and ternary partitions. 
Letting $\textup{Bachet}(m)$ denote the set of Bachet partitions of $m$, 
R{\o}dseth's formula amounts to showing that 
$$|\textup{Bachet}(m)| = f(\frac{1}{2}(3^{n+1}-1)-m)$$ for essentially 
two-thirds of all positive integers $m$. For the other one-third, we will also 
describe what happens, in terms of counting lattice points in polyhedra.

We begin with two observations. 
The first is that we get Bachet partitions from $1+2+7+15$ by sequentially 
peeling off (projecting down) their largest parts:
$
1+2+7+15 \, \longrightarrow \, 1+2+7 \, \longrightarrow \, 1+2 \longrightarrow 1.
$
It's clear in this example that every peeling will project to a unique Bachet partition 
and this {\em hereditary property} is true in general: 
if $\lambda_0 + \lambda_1 + \lambda_2 + \cdots + \lambda_n$ is a Bachet 
partition then so is $\lambda_0 + \lambda_1 + \lambda_2 + \cdots + \lambda_j$ for 
every $j=0,1,\ldots, n-1$. The only tedious part of the proof of this claim is showing 
that $\frac{1}{2}(3^{j}-1) + 1 \leq \lambda_0 + \lambda_1 + \cdots + \lambda_j$.

Secondly, reversing the projection we see that $1+2+7$ could lift to $1+2+7+15$ 
but could also lift to a Bachet partition $1+2+7+16$ of $26$. But what we can 
say is the following: 
if $m^\prime = \lambda_0 + \lambda_1 + \cdots + \lambda_{n-1}$ is a Bachet partition of $m^\prime$ 
(implicit in this statement is that $n-1 = \lfloor \textup{log}_3(2m^\prime) \rfloor$) then 
we can extend it to a Bachet partition $\lambda_0 + \lambda_1 + \cdots + \lambda_{n-1} + (m - m^\prime)$
of a fixed $m$ if and only if 
\begin{center}
(i) $\lambda_{n-1} \leq m - m^\prime$, 
(ii) $m$ has the property that $n = \lfloor \textup{log}_3(2m) \rfloor$ \\ 
and (iii) $m - m^\prime \leq 1 + 2m^\prime$ or $\lceil \frac{m-1}{3} \rceil \leq m^\prime$.
\end{center}

In the case of the Bachet partitions for $m=25$ below, the boldfaced largest terms are peeled 
off to leave precisely the Bachet partitions for $m^\prime = 8,9,10,11,12$ and $13$ and, by the 
projections of the hereditary property, no other Bachet partitions can be built upon to provide 
Bachet partitions for $m=25$. 
$$
\begin{array}{cccccccccc}
 25 & & = & 1+3+9+{\bf 12} & & = & 1+3+8+{\bf 13} & & = & 1+3+7+{\bf 14} \,  \\
    & & = & 1+3+6+{\bf 15} & & = & 1+3+5+{\bf 16} & & = & 1+3+4+{\bf 17} \,  \\
    & & = & 1+2+7+{\bf 15} & & = & 1+2+6+{\bf 16} & & = & 1+2+5+{\bf 17} 
\end{array}
$$
In other words, 
$
|\textup{Bachet}(25)| \, = \, 
\sum_{m^\prime = 8}^{13}
{|\textup{Bachet}(m^\prime)|} = 9 = f(40-25) = f(15).
$
We claim that $|\textup{Bachet}(m)| = f(\frac{1}{2}(3^{n+1}-1)-m)$ holds whenever 
$m$ is sandwiched by 
$$
\frac{1}{2}(3^n -1) + 3^{n-1} \leq m \leq \frac{1}{2}(3^{n+1} -1).
$$

Let's refer to such $m$'s as simply being {\em sandwiched}. 
For sandwiched $m$'s, condition (ii) above is immediately taken care of. We already have 
$m^\prime \leq \frac{1}{2}(3^n-1)$ and $\lambda_{n-1} \leq 3^{n-1}$ which implies that 
$m - m^\prime \geq (\frac{1}{2}(3^n-1) + 3^{n-1}) - \frac{1}{2}(3^n-1) = 3^{n-1} \geq \lambda_{n-1}$ 
as needed for condition (i). And in order that condition (iii) is met we need simply to insist that those 
Bachet$(m^\prime)$'s that extend to Bachet partitions of $m$ are exactly those in the range 
$ 
\lceil \frac{m-1}{3} \rceil \leq 
m^\prime \leq 
\frac{1}{2}(3^n -1).
$
Hence
$$
\textup{Bachet}(m) 
\, = \, 
\bigcup_{m^\prime = \lceil \frac{m-1}{3} \rceil}^{\frac{1}{2}(3^n -1)}
{\textup{Bachet}(m^\prime)}.
$$

By the projecting and lifting of the hereditary property, each Bachet partition 
of $m^\prime$ is extended to a {\em unique} Bachet partition of $m$. 
Hence, 
the number of elements in the above union equals the sum of the number of elements 
in each $\textup{Bachet}(m^\prime)$ of that union. So whenever $m$ is sandwiched we have 
$$
|\textup{Bachet}(m)| 
\, = \, 
\left|
\bigcup_{m^\prime = \lceil \frac{m-1}{3} \rceil}^{\frac{1}{2}(3^n -1)}
{\textup{Bachet}(m^\prime)}
\right| 
\, = \, 
\sum_{m^\prime = \lceil \frac{m-1}{3} \rceil}^{\frac{1}{2}(3^n -1)}
{|\textup{Bachet}(m^\prime)|}.
$$

We are now ready to tie together Bachet partitions and ternary partitions. We claim that 
$f(\frac{1}{2}(3^{n+1}-1) - m) = |\textup{Bachet}(m)|$ for all sandwiched 
$m$'s. We do so once again by induction on $n = \lfloor \textup{log}_3(2m) \rfloor$. The claim holds 
for $n=1$ since $2=1+1; \, 3 = 1+2; \, 4= 1+3$ and also for all sandwiched $m$'s for $n=2$, 
as can be seen here:
\begin{center}
$
7  =  1+3+3  =  1+2+4  =  1+1+5; \, 
8  =  1+3+4  =  1+2+5; \, 
9  =  1+3+5  =  1+2+6; 
$

$
10  =  1+3+6  =  1+2+7; \, 
11  =  1+3+7; \, 
12  =  1+3+8; \, 
13  =  1+3+9.
$ 
\end{center}
So assume that $m$ is sandwiched with $n = \lfloor \textup{log}_3(2m) \rfloor$. We already know 
that 
$
|\textup{Bachet}(m)| \, = \, 
\sum_{m^\prime = \lceil \frac{m-1}{3} \rceil}^{\frac{1}{2}(3^n -1)}
{|\textup{Bachet}(m^\prime)|}
$
and we remark that every such $m^\prime$ in this summation is also sandwiched, but with 
$n-1 = \lfloor \textup{log}_3(2m^\prime) \rfloor$. Hence, by our inductive hypothesis, 
$f(\frac{1}{2}(3^{n}-1) - m^\prime) = |\textup{Bachet}(m^\prime)|$ and 
$$
|\textup{Bachet}(m)| \, = \, 
\sum_{m^\prime = \lceil \frac{m-1}{3} \rceil}^{\frac{1}{2}(3^n -1)}
{f(\frac{1}{2}(3^{n}-1) - m^\prime)}
\, = \, 
f(0)+f(1)+f(2)+ \cdots + f(\frac{1}{2}(3^{n}-1) - \left\lceil \frac{m-1}{3} \right\rceil).
$$
But the input of the last term 
$
\frac{1}{2}(3^{n}-1) - \lceil \frac{m-1}{3} \rceil
$
simplifies to 
$
\left\lfloor  \frac{ \frac{1}{2}(3^{n+1}-1) -m }{3} \right\rfloor
$ 
and so we have 
$$
|\textup{Bachet}(m)| \, = \, 
\sum_{i=0}^{\lfloor  \frac{ \frac{1}{2}(3^{n+1}-1) -m }{3} \rfloor}
{f(i)}
$$
This is exactly the recurrence relation, with the initial conditions still intact, 
that we had hoped to obtain. Hence when $m$ is sandwiched, the generating function for 
the Bachet partitions is exactly $F(x)$, the generating function for the ternary 
partitions.

We close this section by describing what happens for those $m$'s that are not sandwiched. 
Using the generating function $F(x)$ we can define another 
$$
G(x) :=  \sum_{k=0}^{\infty}{g(k)x^k} \, = \,  
\sum_{j=0}^{\infty} { \frac{x^{{3^j}-1}}{1 - x^{2\cdot 3^j}}F(x^{5 \cdot 3^j}) \prod_{i=0}^{j}\frac{1}{1-x^{3^i}}}.
$$ 
Then, adapting the convention that $g(k) = 0$ if $k$ is a negative integer, 
R{\o}dseth's formula claims that the number of Bachet partitions of $m$ equals 
$$
f(\frac{1}{2}(3^{n+1}-1) - m) - g( \frac{1}{2}(3^{n}-1) +3^{n-1} - 1 - m)
$$
Note that $m$ is sandwiched precisely when the input for $g(\cdot)$ is negative. 
An example of a non-sandwiched $m$ is $16$ and for this we have 
$|\textup{Bachet}(16)| =  f(24) - g(5) = 18 - 6$ (we only have to work 
out the first two parts $j=0,1$ of the infinite sum for $g(5)=6$). We can list the 
Bachet partitions for $16$, using Theorem~\ref{thm:necc}:
$$
\begin{array}{ccccccccc}
 16 & = & 1+3+3+9 \,& = & 1+3+4+8 \,& = & 1+3+5+7 \,& = & 1+3+6+6 \,  \\
    & = & 1+2+6+7 \,& = & 1+2+5+8 \,& = & 1+1+5+9 \,& = & 1+2+4+9 \,  \\
    & = & \, 1+1+4+10 & = & \, 1+2+3+10 & = & \, 1+2+2+11 & = & \, 1+1+3+11
\end{array}
$$
Sticking to our promise not to prove R{\o}dseth's formula we will instead present a 
polyhedral picture as to why the generating function $G(x)$ is needed when $m$ is not 
sandwiched. In a previous section we saw a figure showing the Bachet partitions for $m=25$ 
as lattice points in a triangle. A key observation here, and something that holds for 
all sandwiched $m$'s, is that the ordering inequalities 
$\lambda_0 \leq \lambda_1 \leq \cdots \leq \lambda_n$ were subsumed by the 
$\lambda_i \leq 1 + 2 (\lambda_0 + \cdots + \lambda_{i-1})$. This is a significant part of 
what made counting $\textup{Bachet}(m)$ for sandwiched $m$'s relatively straightforward: we 
get the ordering of the parts for free from the other inequalities! Not so when $m$ is not 
sandwiched and we can see this in the following figure for $m=16$:
\begin{center}
\begin{picture}(0,0)%
\includegraphics{twoScale16.pstex}%
\end{picture}%
\setlength{\unitlength}{1450sp}%
\begingroup\makeatletter\ifx\SetFigFont\undefined%
\gdef\SetFigFont#1#2#3#4#5{%
  \reset@font\fontsize{#1}{#2pt}%
  \fontfamily{#3}\fontseries{#4}\fontshape{#5}%
  \selectfont}%
\fi\endgroup%
\begin{picture}(20706,4923)(2821,-5494)
\put(12601,-2761){\makebox(0,0)[lb]{\smash{{\SetFigFont{20}{24.0}{\rmdefault}{\mddefault}{\updefault}{\color[rgb]{0,0,0}$-$}%
}}}}
\put(14986,-5416){\makebox(0,0)[lb]{\smash{{\SetFigFont{5}{6.0}{\rmdefault}{\mddefault}{\updefault}{\color[rgb]{0,0,0}$\lambda_2 \leq 3 + 2\lambda_1$}%
}}}}
\put(14311,-4381){\makebox(0,0)[lb]{\smash{{\SetFigFont{5}{6.0}{\rmdefault}{\mddefault}{\updefault}{\color[rgb]{0,0,0}$\lambda_1 \leq 3$}%
}}}}
\put(14041,-3886){\makebox(0,0)[lb]{\smash{{\SetFigFont{5}{6.0}{\rmdefault}{\mddefault}{\updefault}{\color[rgb]{0,0,0}$\lambda_1 \geq 1$}%
}}}}
\put(19531,-5056){\makebox(0,0)[lb]{\smash{{\SetFigFont{5}{6.0}{\rmdefault}{\mddefault}{\updefault}{\color[rgb]{0,0,0}$\lambda_1 \leq \lambda_2$}%
}}}}
\put(3511,-5416){\makebox(0,0)[lb]{\smash{{\SetFigFont{5}{6.0}{\rmdefault}{\mddefault}{\updefault}{\color[rgb]{0,0,0}$\lambda_2 \leq 3 + 2\lambda_1$}%
}}}}
\put(2836,-4381){\makebox(0,0)[lb]{\smash{{\SetFigFont{5}{6.0}{\rmdefault}{\mddefault}{\updefault}{\color[rgb]{0,0,0}$\lambda_1 \leq 3$}%
}}}}
\put(6211,-1231){\makebox(0,0)[lb]{\smash{{\SetFigFont{5}{6.0}{\rmdefault}{\mddefault}{\updefault}{\color[rgb]{0,0,0}$\lambda_3 \leq 3+2\lambda_1 + 2\lambda_2$}%
}}}}
\put(17686,-1231){\makebox(0,0)[lb]{\smash{{\SetFigFont{5}{6.0}{\rmdefault}{\mddefault}{\updefault}{\color[rgb]{0,0,0}$\lambda_3 \leq 3+2\lambda_1 + 2\lambda_2$}%
}}}}
\put(16561,-1726){\makebox(0,0)[lb]{\smash{{\SetFigFont{5}{6.0}{\rmdefault}{\mddefault}{\updefault}{\color[rgb]{0,0,0}$\lambda_2 \leq \lambda_3$}%
}}}}
\end{picture}%

\end{center}
The $18$ points on the left indicate the lattice points in the polyhedron defined by 
$\lambda_1 \leq 3, \, \lambda_2 \leq 3 + 2\lambda_1 \, \textup{\&} \,  \lambda_3 \leq 3 + 2\lambda_1 + 2 \lambda_2$ 
and living in the plane defined by $\lambda_0 = 1$ and $\lambda_0 + \lambda_1 + \lambda_2 + \lambda_3 = 16$. 
This polyhedron pays no regard to the ordering of the parts in the Bachet partitions. This is precisely what 
$f(40 - 24) = 18$ is counting. On the right we see the effect of including the inequalities that define the 
ordering on the parts $\lambda_0 \leq \lambda_1 \leq \lambda_2 \leq \lambda_3$: we need to subtract 
exactly $g(5) = 6$ lattice points from those on the left to get the count of 
$|\textup{Bachet}(16)| = 12$ just right! 

We motivated why R{\o}dseth's general formula works by explaining it in terms of counting lattice points 
in polyhedra. This would not be the usual approach in integer partitions owing to the fact that 
partitions can rarely be described in terms of lattice points in a polyhedron. Many of the wonderful 
results in integer partitions depend largely on the ability to manipulate generating functions 
in much the same way that we attained the functional equation $(1-x)F(x) = F(x^3)$. 
However, it is no 
fluke that the generating functions above counted the lattice points and this method of counting lattice 
points in a polyhedron (by generating functions) is one of the most beautiful and effective methods 
for solving the general problem of counting lattice points in polyhedra. The textbook of Beck \& Robins 
\cite{BecRob} provides a wonderful, accessible introduction to this problem and 
how it arises in many contexts like discrete geometry, 
number theory and combinatorics. 

\section*{Other Generalizations: One-Scale \& Error-Correcting Bachet's Problems}

There are two natural variants of Bachet's problem. The first is what if we are only allowed to 
place weights on one side of the scale pan. The second is that of discerning an integer value that is unknown. 
In other words: 

\noindent {\em The one-scale Bachet problem}: What is the least number of pound weights that can be 
used on a scale pan to weigh any integral number of pounds from 1 to m inclusive, if the weights 
can be placed in only one of the scale pans~?

\noindent {\em The error-correcting Bachet problem}: Given a fixed integer weight of unknown weight $l$, weighing 
no more than m pounds, what is the least number of pound weights that can be used on a scale pan to discern $l$'s 
value, if the weights can be placed in either of the scale pans~?

These, and the original Bachet problem, lead to the following definition \cite{BruO'Sh}
\begin{definition}
A partition $m = \lambda_0 + \lambda_1 + \cdots + \lambda_n$ with the parts in increasing order 
is an {\em $e$-relaxed $r$-complete partition} ({\em $(e,r)$-partition} for short) if 
no $e+1$ consecutive integers between $0$ and $rm$ are absent from the set 
$\{ \sum_{i=0}^{n} \alpha_i \lambda_i : \alpha_i \in \{ 0,1,\ldots,r \} \}$. We call the 
partition {\em minimal} if $n$ is as small as possible with this property.
\end{definition}

The original Bachet problem is the study of minimal $(0,2)$-complete partitions. The one-scale 
variant is that of minimal $(0,1)$-partitions. Note that for $m=15$, there is a unique solution 
given by $15 = 1+2+4+8$. There is also unique solution to the error-correcting 
variant for $m=80$ given by $2 + 6 + 18 + 54$. Note, for example, that if the weight we need to discern is 
$l = 5$ then we need not weigh $l=5$ precisely, only to observe that $l$ is heavier than $4 = 6-2$ and 
lighter than $6$. In other words, no two consecutive $l$'s are absent from the set of integers achievable 
with the parts of the partition using both of the scale pans. Thus the error-correcting Bachet problem as 
stated above is simply that of the minimal $(1,2)$-partitions. 

The problem of classifying and enumerating the minimal $(e,r)$-partitions are completely understood in 
much the same manner as we did for Bachet's problem. To begin with, for a minimal $(e,r)$-partition 
$m = \lambda_0 + \lambda_1 + \cdots + \lambda_n$ it's not too hard to see that we get 
$\lambda_0 \leq e+1$ and that $\lambda_i \leq (e+1) + r \sum_{j=0}^{i-1}{\lambda_j}$ for all $i \leq n$. 
While the error term $e$ does affect the precise count 
the minimal $(e,r)$-partitions, the $(r+1)$-ary partitions are still the dominant player for enumerating 
these partitions, with the error term making only a minor impact. As we might expect {\em $(r+1)$-ary partitions} 
are partitions of integers whose parts are powers of $r+1$.

The full story of these variants can be found as follows: 
the description of the one-scale problem was first described by Brown \cite{Bro} and extended to the 
$(0,r)$-partitions by Park \cite{Par}, who called them {\em minimal $r$-complete partitions}. The one-scale 
problem was enumerated for sandwiched $m$'s by binary partitions in \cite{Osh}, and 
R{\o}dseth \cite{RodMpart} extended this enumeration for all minimal $(0,1)$-partitions. R{\o}dseth went 
further and enumerated all $(0,r)$-partitions in \cite{Rod}. With Bruno, we extended all these arguments to the minimal 
$(e,r)$-partitions \cite{BruO'Sh}. Park's expressed motivation in \cite{Par} was to complement the 
{\em perfect partitions} of MacMahon from the 1880's, to which we turn our attention to next in our last section.

\section*{MacMahon's Perfect Partitions}

In 1886 Major Percy A. MacMahon \cite{MacQuar} \cite[pp. 217--223]{MacComb} 
proposed and solved an alternative generalization 
to Bachet's problem, which differs significantly from the generalization that we have 
investigated until now. It's appropriate too that we should include MacMahon's 
contribution to Bachet's problem since, as Gian-Carlo Rota persuasively argues in 
his introduction to MacMahon's collected papers \cite{MacColl}, MacMahon's 
substantial contributions to the foundation of modern combinatorics have not always 
been given their proper due. 

\parbox{1.7in}{\includegraphics[scale=0.2]{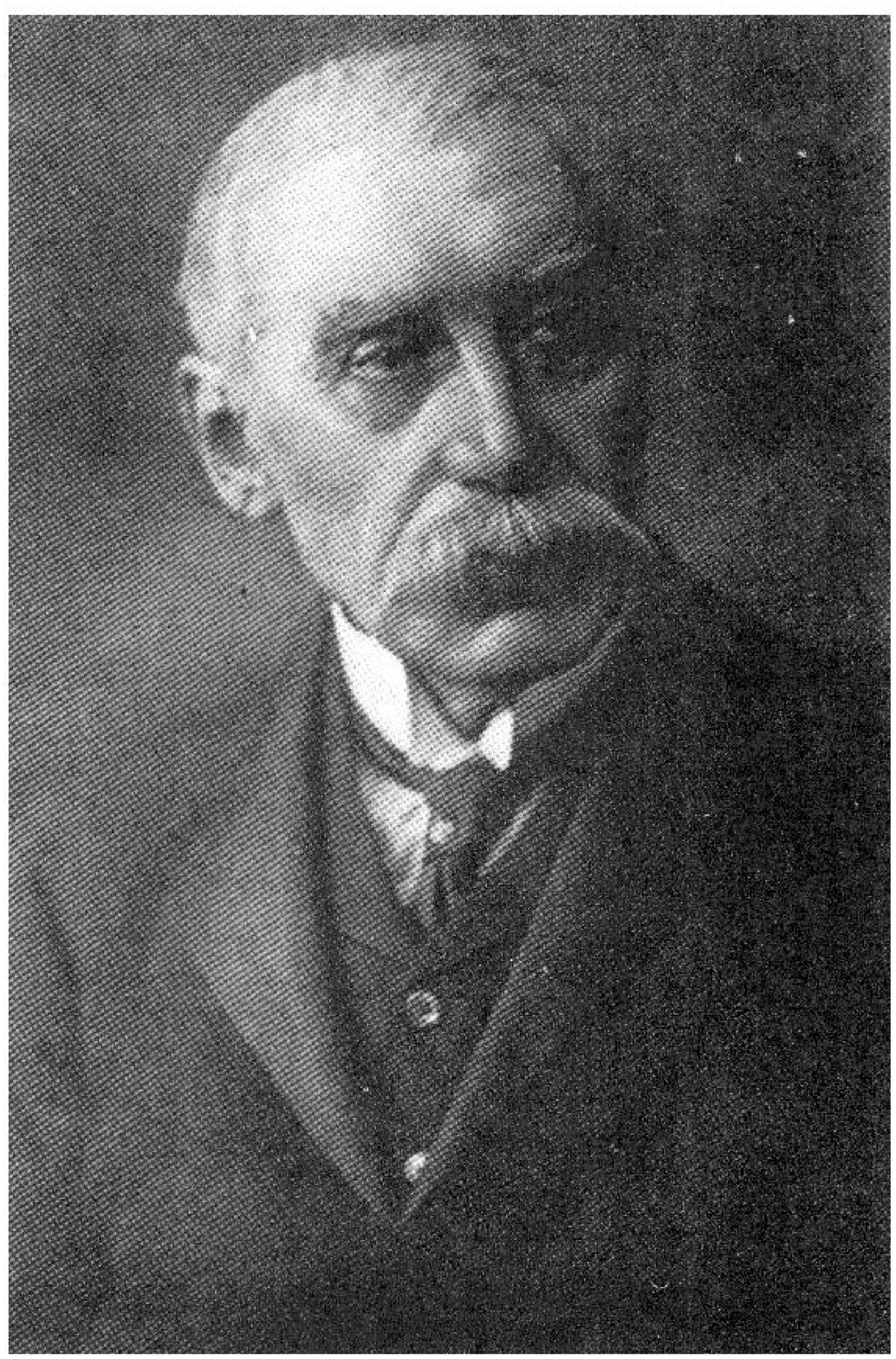}
}
\parbox{4.6in}{
MacMahon noted that the example of $40 = 1+3+9+27$ had the property that every integer 
weight $l$ between $1$ and $40$ can be weighed in a {\em unique} manner using the weights $1,3,9$ and $27$ 
on a two-scale pan. In other words, MacMahon discarded the {\em minimality of parts} condition 
that we focused on here and instead added the {\em uniqueness} condition. The set of all partitions for 
MacMahon's generalization of Bachet's problem for $40$ is thus:
$
\underbrace{1+1+ \cdots +1+1}_{40 \, \textup{times}}, \, 
\underbrace{1+1+ \cdots +1+1}_{13 \, \textup{times}} + 27 , \, 
1+ \underbrace{3+3+ \cdots +3+3}_{13 \, \textup{times}}, \, 
1+1+1+1+9+9+9+9, \, 1+1+1+1+9+27, \, 1+3+3+3+3+27, \, 1+3+9+9+9+9, \, 1+3+9+27.
$
For shorthand, we write these partitions respectively as 
}

$$
(1)^{40}, \, (1)^{13}+27, \, 1 + (13)^3, \, (1)^4+(9)^4, \, 
(1)^4+9+27, \, 1+(3)^4+27, \, 1+3+(9)^4, \, 1+3+9+27.
$$
Note that we view the repeated parts of such partitions as indistinguishable as in the 
case of $(1)^{13}+27$ we regard there being a unique expression for $4$ as $1+1+1+1$; 
not ${13}\choose{4}$ distinct expressions! Also note that it includes our unique Bachet 
partition for $m = 40$. To describe all such partitions it will be easier to begin 
by analyzing the one-scale analogue of MacMahon's generalization, from which the 
two-scale problem will follow immediately.

Starting with the one-scale problem, MacMahon called a partition 
$m = \lambda_0 + \lambda_1 + \cdots + \lambda_s$ {\em perfect} 
if every $l$ between $0$ and $m$ can be written uniquely 
as $l = \sum_{i = 0}^s{\alpha_i \lambda_i}$ for $\alpha_i \in \{ 0,1 \}$ 
(and with repeated parts regarded as indistinguishable as described above). For example, 
we have $8$ perfect partitions for $m=11$:
$$
(1)^{11}, \, (1)^5 + (6)^1, \, (1)^1+(2)^5, \, (1)^3+(4)^2, \, 
(1)^2+(3)^2, \, (1)^2+(3)^1+(6)^1, \, (1)^1+(2)^2+(6)^1, \, (1)^1+(2)^1+(4)^2
$$
MacMahon's insight was that the perfect partitions of $m$ are in bijection with the 
{\em ordered factorizations} of the integer $m+1$ -- the set of all possible 
factorizations (not including multiples of 1) of $m+1$ but where we account 
for order too. For example, the set of ordered factorizations for $12$ equals 
$$
12, \,\, 6 \times 2, \,\, 2 \times 6, \,\, 4 \times 3, \,\, 
3 \times 4, \,\, 3 \times 2 \times 2, \,\, 2 \times 3 \times 2, \,\, 2 \times 2 \times 3.
$$

\begin{theorem} \textup{(MacMahon \cite{MacQuar})} \label{thm:perfect}
The perfect partitions of $m$ are in bijection with the 
ordered factorizations of $m+1$.
\end{theorem}
\begin{proof}
Consider the ordered factorization $f_1 \times f_2 \times f_3 \times \cdots \times f_r$ of $m+1$. 
From this factorization consider the recursively defined partition of $m$:
$$
(1)^{f_1-1} + (f_1)^{f_2-1} 
+ (f_1 \cdot f_2)^{f_3-1} 
+ (f_1 \cdot f_2 \cdot f_3)^{f_4-1} 
+ 
\cdots 
+ 
(f_1 \cdot f_2 \cdots f_{i-1})^{f_{i}-1} 
+ 
\cdots 
+ 
(f_1 \cdot f_2 \cdots f_{r-1})^{f_r-1}
$$ 
This partition sums to $m$ since the parts form a telescopic sum 
which collapses to $f_1 \cdot f_2 \cdot f_3 \cdots f_r - 1$. 

To see that the partition is perfect note that the parts $(1)^{f_1-1}$ suffice 
to expresses every $1 \leq l_1 \leq f_1-1$. Since all other parts of the partition 
are larger than $f_1$ then the sum of $l_1$ $1$'s is the unique way to express $l_1$ 
using parts of the above partition. Next, each $l_2 = f_1, 2f_1, \ldots, (f_1-1)(f_2)$ 
can be expressed by the parts $(f_1)^{f_2-1}$. And, as argued above, using the 
parts $(1)^{f_1-1}$ we can express every integer less than or equal to 
$(f_1 \cdot f_2-1)$ uniquely as $l_1+l_2$ where $l_1$ and $l_2$ are one of those 
above. Since the other parts of the partition are larger than $f_1 \cdot f_2-1$ then these 
are the unique expressions for every integer less than or equal to $f_1 \cdot f_2-1$. Now 
repeat the argument; rigorously one needs to complete the argument by induction 
on the number of terms in the ordered factorization.

With much the same argument in mind, it's not too difficult to see that any perfect 
partition 
$m = (\lambda_0)^{g_1-1} + (\lambda_1)^{g_2-1} + (\lambda_2)^{g_3-1} + \cdots + (\lambda_{s-1})^{g_s-1}$ must provide an ordered factorization of $m+1$: there must be at least 
one part of the partition equal to $1$ since we must be able to express $1$ as a subsum of 
the parts of the partition. Hence $\lambda_0 = 1$. Next $g_1 = \lambda_1$: if $\lambda_1 < g_1$ 
then the integer $\lambda_1 = g_1 + (1)^{\lambda_1 - g_1} = (1)^\lambda_1$ which would upset 
the uniqueness property. And if $\lambda_1 > g_1$ then we would have no way of expressing $g_1$ 
as a subsum of the parts of the partition. Hence, $\lambda_1 = g_1$. Assuming then that 
$\lambda_1$ is repeated $g_2-1 \geq 1$ times we are then, by a similar argument, forced to have 
$\lambda_2 = g_1 \cdot g_2$ and we can repeat this argument to yield 
$\lambda_i = g_1 \cdot g_2 \cdots g_i$ and this perfect partition yields an ordered factorization 
of $m+1$ vis-a-vis $g_1 \times g_2 \times \cdots \times g_s$.
\end{proof}
The perfect partitions for $11$ and the unique factorizations of $12$ listed above are done 
so in the order of the bijection between the sets. For example 
$12 \leftrightarrow (1)^{12-1} = 1 + 1 + \cdots + 1$ and 
$2 \times 3 \times 2 \leftrightarrow (1)^{2-1} + (2)^{3-1} + (2 \cdot 3)^{2-1} = 1 + 2 + 2 + 3 + 3$. 
We can't say for certain but would be willing to wager that MacMahon's motivation for calling 
these partitions ``perfect'' would be that the confluence between factorizations 
and sums reminded him of a similar confluence seen in perfect numbers.

With this characterization of perfect partitions we can in turn solve the 
{\em two-scale} problem -- 
these are what MacMahon called {\em subperfect partitions}. MacMahon called a partition 
$m = \lambda_0 + \lambda_1 + \cdots + \lambda_s$ {\em subperfect} 
if every $l$ between $0$ and $m$ can be written uniquely 
as $l = \sum_{i = 0}^s{\alpha_i \lambda_i}$ for $\alpha_i \in \{ -1,0,1 \}$ 
(and with repeated parts regarded as indistinguishable as in the case of perfect partitions). 
\begin{theorem} \cite[\S 3]{MacQuar}
The ordered factorizations of $2m+1$ are in bijection with the subperfect partitions of $m$.
 \end{theorem}
\begin{proof}
Consider any ordered factorization of $2m+1 = f_1 \times f_2 \times f_3 \times \cdots \times f_r$. Since $2m+1$ 
is odd then each $f_i \geq 3$ and must also be odd. In turn, each $f_i-1 \geq 2$ and is even. 
From Theorem~\ref{thm:perfect} we have a perfect partition of $2m$ given by 
$$
(1)^{f_1-1} + (f_1)^{f_2-1} 
+ (f_1 \cdot f_2)^{f_3-1} 
+ (f_1 \cdot f_2 \cdot f_3)^{f_4-1} 
+ 
\cdots 
+ 
(f_1 \cdot f_2 \cdots f_{i-1})^{f_{i}-1} 
+ 
\cdots 
+ 
(f_1 \cdot f_2 \cdots f_{r-1})^{f_r-1}
$$ 
By definition, every $l$ between $0$ and $2m$ can be expressed in a unique 
way as a subsum of these parts. In other words, as a $\{ 0,1 \}$-combination 
of the parts of this perfect partition. However, each one of the $f_i -1$'s are even 
and so every indistinguishable part appears an even number of times in the partition 
and so we can rephrase the above partition being a perfect partition for $2m$ as 
$$
(1)^{\frac{f_1-1}{2}} + (f_1)^{\frac{f_2-1}{2}} 
+ (f_1 \cdot f_2)^{\frac{f_3-1}{2}} 
+ (f_1 \cdot f_2 \cdot f_3)^{\frac{f_4-1}{2}} 
+ 
\cdots 
+ 
(f_1 \cdot f_2 \cdots f_{i-1})^{\frac{f_{i}-1}{2}} 
+ 
\cdots 
+ 
(f_1 \cdot f_2 \cdots f_{r-1})^{\frac{f_r-1}{2}}
$$
is a 2-complete partition of $m$ with MacMahon's uniqueness property preserved. But as 
we have noted in earlier sections, 2-complete partitions are exactly the Bachet 
partitions without the minimality of parts constraint. Since the uniqueness property 
is also preserved then we have shown that the subperfect partitions of $m$ are given 
precisely by the ordered factorizations of $2m+1$.
\end{proof}
The eight subperfect partitions of $40$ that we opened this section with 
are attained respectively from the ordered factorizations of $81$:
$$
81, \, 27 \times 3, \, 3 \times 27, \, 9 \times 9 , \, 
9 \times 3 \times 3, \, 3 \times 9 \times 3, \, 3 \times 3 \times 9, \, 3 \times 3 \times 3 \times 3.
$$ 

All in all, we get the MacMahon's two-scale problem for (almost) free from the one-scale 
problem and the connection to between (additive) partitions and (multiplicative) factorizations is 
surprising and satisfying. As claimed in the introduction, the original Bachet partition of 
$40 = 1+3+9+27$ comes from the factorization $3 \times 3 \times 3 \times 3$ of $81$.

Permit us to finish this article not so much with a criticism of Bachet partitions but by expressing 
a yearning. While the Bachet partitions span from a highly accessible problem with their general solutions 
both elegant and succinct, we know of no other connection to other seemingly unrelated 
areas of mathematics, even on the basic level like that which takes place between perfect 
partitions and factorizations. They do however enjoy a distinguishing feature that fits into a theme 
that currently enjoys some prominence: Bachet partitions can be described in terms 
of inequalities on the parts, a feature shared by the {\em lecture hall partitions} of 
Bousquet-M{\' e}lou \& Eriksson \cite{BosEri} and by the {\em symmetrically constrained compositions} 
of \cite{BecGesLeeSav}. Let us hope for further developments on this theme of polyhedral descriptions of 
integer partitions!

\section*{Acknowledgements} 
This work was carried out while I was visiting NUI Galway for the 2009--2010 academic year.
Thanks to Jorge Bruno, Graham Ellis, David Quinn and Jerome Sheahan of Galway, for allowing me to regale 
them with all things Bachet during that time.

\end{document}